\documentclass[12pt]{article}
\usepackage{amsmath}

\usepackage[utf8]{inputenc}
\usepackage[english]{babel}
\usepackage{amsmath, amsfonts, amsthm, amssymb, setspace, enumitem, tikz-cd, bbm, mathtools, graphicx,scalerel, mathrsfs, titling, hyperref}
\usepackage{graphicx}
\usepackage[toc,page]{appendix}
\usepackage{mathrsfs}
\usepackage{comment}
\usepackage{setspace}
\usepackage{arydshln}

\theoremstyle{break}
\newtheorem{theorem}{Theorem}[section]
\newtheorem{theorem-break}[theorem]{Theorem}
\newtheorem{lemma}[theorem]{Lemma}
\newtheorem{corollary}[theorem]{Corollary}
\newtheorem{proposition}[theorem]{Proposition}
\newtheorem{remark}[theorem]{Remark}

\newtheorem{definition}[theorem]{Definition}

\numberwithin{equation}{section}
\numberwithin{figure}{section}

\newcommand{\tr}{\mathrm{tr}}
\newcommand{\ind}{\mathrm{ind}}
\title{The Helton--Howe measure of almost normal Toeplitz operators}
\author{Yuto Sugahara \\
Graduate School of Science \\
Kyoto University \\
Sakyo-ku, Kyoto 606-8502, Japan \\
\href{mailto:kyodaisugaharay@gmail.com}{kyodaisugaharay@gmail.com}
}
\begin{document}

\setlength{\abovedisplayskip}{5pt}   
\setlength{\belowdisplayskip}{5pt}  

\begin{titlepage}
    \maketitle
    \begin{abstract}
The Helton--Howe measure associated with an almost normal operator was constructed by Helton and Howe. It provides a trace formula that allows us to calculate the trace of commutators that would otherwise be incalculable. We will investigate almost normal Toeplitz operators and determine their Helton--Howe measures in terms of the harmonic extensions of the symbols. Our result may be thought of as a generalization of the winding number formula of the Fredholm index of a Toeplitz operator.
    \end{abstract}

\end{titlepage}

\clearpage

\section{Introduction}
Let $T \in \mathcal{B}(\mathcal{H})$ be a bounded operator with the trace-class self-commutator $[T^*,T]$. We say that such an operator is almost normal. In \cite{heltonhowe1973}, Helton and Howe constructed the Helton--Howe measure $P_T$ for an almost normal operator $T$ and obtained the following significant trace formula;
$$\tr([p(X,Y), q(X,Y)])=\int_{\sigma(T)} J(p,q)dP_T,$$
where $p,q$ are polynomials and $J(p,q)=\frac{\partial{p}}{\partial{x}} \frac{\partial{q}}{\partial{y}} - \frac{\partial{q}}{\partial{x}} \frac{\partial{p}}{\partial{y}}$.
On the other hand, Pincus constructed an $L^1$ function called the principal function rather than a measure in \cite{Pincus1968} earlier than \cite{heltonhowe1973}. Later in \cite{cpexponentialformula}, Carey and Pincus showed that the Helton--Howe measure and the principal function agree and hence that the Helton--Howe measure is absolutely continuous with respect to the Lebesgue measure. Furthermore, in \cite{pincusmosaics}, they generalized the above trace formula to a von Neumann algebra with a semifinite trace. 

Thanks to the Helton--Howe trace formula, if we determine $P_T$, we can calculate the trace of an operator which is hard to compute by other methods. However, there are only a few kinds of almost normal operators whose Helton--Howe measures are well-described. Since the Helton--Howe measure is absolutely continuous, we may determine $P_T$ if we have a plenty of information on the spectrum of $T$. To the best of the author's knowledge, \cite{Carey1979MeanMotion} is the single reference that determines the Helton--Howe measures of almost normal operators without rich spectral description. In Theorem 6 of that paper, it was shown that the Helton--Howe measure of some Toeplitz operators in a type II$_\infty$ factor is constant times the mean motion.

The goal of this paper is to determine the Helton--Howe measure of almost normal Toeplitz operators on the Hardy space $H^2(\mathbb{T})=\{ \phi \in L^2(\mathbb{T})~|~ \hat{\phi}(k)=0~~ \forall k<0 \}$. If $\phi$ is a smooth function on the unit circle, the essential spectrum of $T_\phi$ is $\phi(\mathbb{T})$, which has 0 Lebesgue measure. Hence by the index formula for the Helton--Howe measure established in \cite{heltonhowe1973},
$$dP_{T_\phi}=-\frac{1}{2\pi i} \ind(T_\phi-x-iy)dxdy$$
holds with $\ind(T)$ being the index of a Fredholm operator $T$.
We know that if $x+iy$ is not in the essential spectrum of $T_\phi$, $\ind(T_\phi-x-iy)$ is equal to $-\mathrm{wind}(\phi, x+iy)$, minus the winding number of $\phi$ around $x+iy$. Moreover, let $\Phi$ be the harmonic extension of $\phi$ on the unit disc and $J(\Phi)$ be the Jacobian of the mapping $\Phi: \mathbb{D}\to \mathbb{R}^2$. Then, $\mathrm{wind}(\phi, x+iy)$ equals the following signed multiplicity function.
$$m_\Phi(x+iy)=\sum_{\Phi(z)=x+iy, |z|<1} \mathrm{sgn}(J(\Phi)(z)).$$
So, the Helton--Howe measure of a Toeplitz operator with a smooth symbol is determined by $m_\Phi$. 

 The main theorem of this paper shows that the similar result holds for an almost normal Toeplitz operator with a non-smooth symbol $\phi$. To state the result, we use an approximation by $\phi_r$'s, the convolution of $\phi$ and the Poisson kernel. Theorem \ref{general formula} reveals
$$dP_{T_{\phi}}=\mathrm{weak}^*\text{-}\lim_{r \nearrow1} dP_{T_{\phi_r}}=\mathrm{weak}^*\text{-}\lim_{r \nearrow1} \frac{1}{2\pi i}m_{\Phi_r}(x+iy)dxdy.$$
The limit in the above formula is with respect to the weak-$*$ topology of $C(K)^*$ for some compact set $K\subset\mathbb{C}$. In the case $\int _\mathbb{D}|J(\Phi)|dxdy<\infty$, we can get rid of the limit and obtain
 $$dP_{T_{\phi}}=\frac{1}{2\pi i}m_{\Phi}(x+iy)dxdy.$$
 The integrability of $J(\Phi)$ is ensured if $\Re{\phi}\in B_p$ and $\Im{\phi}\in B_q$ where $p,q>1$ are finite H\"{o}lder conjugate numbers and $B_p:=B_{p,p}^{1/p}(\mathbb{T})$ is the Besov space.
In summary, the Helton--Howe measure of an almost normal Toeplitz operator can be described  somehow in terms of the signed multiplicity function.

\section{Notation and Preliminaries}

In this paper, $\mathcal{H}$ denotes a separable infinite dimensional complex Hilbert space and every operator is in $\mathcal{B(\mathcal{H})}$, the algebra of all the bounded operators. The inner product on $\mathcal{H}$ is linear with respect to the first variable.
For an operator $T$, $\sigma_e(T)$ is the essential spectrum of $T$.
If $T$ is a Fredholm operator, $\ind(T)$ is the Fredholm index of $T$.

We will discuss Schatten $p$-classes and other normed ideals. For a comprehensive reference for normed ideals, see \cite{simon2005trace} or \cite{GohK69}.

\begin{definition}
We denote by $\mathcal{J}_p$ the ideal of Schatten $p$-class operators and by $||\cdot||_p$ the Schatten $p$-norm.
Namely, $\mathcal{J}_1$ is the trace class, and $\mathcal{J}_2$ is the Hilbert--Schmidt class. We denote the trace on $\mathcal{J}_1$ by $\tr:\mathcal{J}_1\to \mathbb{C}.$
\end{definition}

\begin{proposition}[Corollary 3.8 of \cite{simon2005trace}]\label{properties of tr}
If $A,B\in \mathcal{B(\mathcal{H})}$ have the property that both $AB$ and $BA$ are in $\mathcal{J}_1$, then $\tr(AB)=\tr(BA)$.
\end{proposition}

\begin{definition}
Let $\Psi$ be a symmetric norm. We denote the corresponding maximal normed ideal by $(J_\Psi,||\cdot||_{J_\Psi})$. $J_\Psi^{(0)}$ is the separable ideal which is the closure in $J_\Psi$ of finite rank operators. $\Psi^*$ is the conjugate symmetric norm of $\Psi$. 
\end{definition}

Let $\mathbb{D} $ be the unit disc $\{ z\in \mathbb{C} ~|~ |z|<1\}$ and $\mathbb{T}$ be its boundary. We regard $\mathbb{T}$ as a measure space endowed with the normalized Lebesgue measure $\frac{dt}{2\pi}$.

The $k$-th Fourier coefficient of $ \phi \in L^1(\mathbb{T})$ is 
$$\hat{\phi}(k)=\frac{1}{2\pi}\int_0^{2\pi}\phi(e^{it})e^{-ikt}dt.$$

Recall that the Hardy space $H^2(\mathbb{T})=\{ \phi \in L^2(\mathbb{T})~|~ \hat{\phi}(k)=0~~ \forall k<0 \}$ is a closed subspace of $L^2(\mathbb{T})$ and denote by $P_+$ the orthogonal projection from $L^2(\mathbb{T})$ onto $H^2(\mathbb{T})$. 
For $\phi \in L^\infty (\mathbb{T})$, consider the multiplication operator $M_\phi$ on $L^2(\mathbb{T})$. 
The Toeplitz operator $T_\phi$ and the Hankel operator $H_\phi$ in $ B(L^2(\mathbb{T}))$ with the symbol $\phi$ are defined by $T_\phi=P_+ M_\phi P_+$ and by $H_\phi=(1-P_+)M_\phi P_+$, respectively.

For $\phi\in C(\mathbb{T})$ and $\lambda\notin\phi(\mathbb{T})$, $\mathrm{wind}(\phi,\lambda)$ is the winding number of $\phi$ around $\lambda$. A detailed description of the winding number can be found in Chapter 4.4 of \cite{arvesonspectral} in which the notation $\#(\phi-\lambda):=\mathrm{wind}(\phi,\lambda)$ is used.

If $\phi \in L^\infty(\mathbb{T})$, we use the symbol $\phi_r$ to denote the convolution of $\phi$ and the Poisson kernel, and $\Phi$ to denote the harmonic extension of $\phi$ on the unit disc.
That is, $\Phi:\mathbb{D}\to \mathbb{C}$ is defined by 
$$ \Phi(re^{i\theta})=\phi_r(e^{i\theta})=\sum_{k\geq 0}\hat{\phi}(k)r^ke^{ik\theta}+\sum_{k>0}\hat{\phi}(-k)r^ke^{-ik\theta}$$
for $r<1$.
As above, we use small letters to represent functions on the unit circle, while capital letters indicate their harmonic extensions.

\begin{definition}\label{def of m}
Assume $\phi$ is an $L^\infty$ function on $\mathbb{T}$ and that $\#\Phi^{-1}(\{w\})$ is finite for almost all $w\in \mathbb{C}$. Then, we define the signed multiplicity function $m_\Phi$  by
$$m_\Phi(w)=\sum_{\Phi(z)=w, |z|<1} \mathrm{sgn}(J(\Phi)(z)),$$
where $J(\Phi)=\left|\frac{\partial \Phi}{\partial z}\right|^2-\left|\frac{\partial \Phi}{\partial \bar{z}}\right|^2$ is the Jacobian of $\Phi$.
\end{definition}

\begin{lemma}\label{existence criterion}
If
$\int _\mathbb{D}|J(\Phi)|dxdy<\infty$, $m_\Phi$ exists as an $L^1$ function.
\end{lemma}

\begin{proof}
Changing the variables, we have
\begin{align*}
\int_{{\Phi(\mathbb{D})}} |m_{\Phi}(x+iy)|dxdy &\leq \int_{{\Phi(\mathbb{D})}} \left(\sum_{\Phi(z)=x+iy, |z|<1} 1\right)dxdy\\
&=\int _\mathbb{D}|J(\Phi)|dxdy<\infty.
\end{align*}
This implies that $\#\Phi^{-1}(\{w\})<\infty$ almost everywhere, hence $m_\Phi$ is well-defined.
\end{proof}
\begin{corollary}
If $\phi\in C^\infty(\mathbb{T})$, $m_\Phi$ exists as an $L^1$ function.
\end{corollary}

\section{The Helton--Howe Theory}

\subsection{The Helton--Howe Trace Formula}
Below is a brief summary of properties of the Helton--Howe measure.
\begin{definition}
We say that an operator $T$ with the Cartesian decomposition $X+iY$ is almost normal if its self-commutator $[T^*,T]=2i[X,Y]$ is in the trace class. 
\end{definition}

\begin{definition}
For $C^\infty$ functions $f,g$ defined on some region of $\mathbb{R}^2$, define
$$J(f,g):=\frac{\partial{f}}{\partial{x}} \frac{\partial{g}}{\partial{y}} - \frac{\partial{g}}{\partial{x}} \frac{\partial{f}}{\partial{y}}.$$
This is the Jacobian if $f,g$ are real-valued. We are just using the same notation even if functions involved are not real-valued.
\end{definition}

\begin{theorem}[Formula (1') in the main theorem of~\cite{heltonhowe1973}]\label{HH trace formula}
For every almost normal operator $T=X+iY$, there exists a unique regular Borel (purely imaginary) measure $P_T$ on the complex plane supported on $\sigma(T)$ such that
$$\tr([p(X,Y), q(X,Y)])=\int_{\sigma(T)} J(p,q)dP_T$$
where $p,q$ are polynomials in two variables.
\end{theorem}

\begin{definition}
We call $P_T$ in Theorem \ref{HH trace formula} the Helton--Howe measure of $T$. 
\end{definition}

\begin{theorem}[Assertion iii in the main theorem of~\cite{heltonhowe1973}]\label{index formula}
Let $T$ be an almost normal operator.
If $U\subset \mathbb{C}$ is a connected component of $\mathbb{C}\setminus \sigma_e(T)$ and if $\lambda \in U$, then on $U$,
$$
\frac{dP_T}{dxdy}=-\frac{1}{2\pi i}\ind(T-\lambda).$$
\end{theorem}

\begin{theorem}[essentially due to Theorem 2 of \cite{cpexponentialformula}]\label{a.c.}
The Helton--Howe measure $P_T$ is absolutely continuous with respect to the Lebesgue measure.
\end{theorem}

 Theorem \ref{index formula} and Theorem \ref{a.c.} imply that the Helton--Howe measure of a compact almost normal operator is $0$ and that the Helton--Howe measure of the unilateral shift is $\frac{1}{2\pi i}dxdy|_\mathbb{D}.$

\subsection{Estimates of the Total Variation}

We give sophisticated results on the total variation of the Helton--Howe measure.

\begin{definition}
We denote by $||P_T||=|P_T|(\mathbb{C})$ the total variation of $P_T$. 
\end{definition}

\begin{proposition}[A remark by Larry Brown in \cite{heltonhowe1973}]
Let $T$ be an almost normal operator.
The total variation of its Helton--Howe measure is bounded by $\frac{||[T^*,T]||_1}{2}$.
\end{proposition}

This estimate turns out to be loose once we consider a Hilbert--Schmidt operator with a nontrivial self-commutator. But since the Helton--Howe measure is invariant under trace-class perturbations, we have a better estimate.

\begin{corollary}\label{better estimation}
$$2||P_T||\leq \inf_{T-T'\in\mathcal{J}_1}||[(T')^*,T']||_1.$$
\end{corollary}

This inequality is strict in the following elementary cases.

\begin{enumerate}
\item Hilbert--Schmidt operators.\\
Let $T$ be a Hilbert--Schmidt operator and $\{P_n\}_{n=1}^\infty$ be finite rank projections converging in the SOT to the identity operator. 
Then, it is easy to see that 
$$\lim_{n\to \infty}||[(T(1-P_n))^*,T(1-P_n)]||_1=0.$$
On the other hand, $P_T=0$ as stated in the remark following Theorem \ref{a.c.}.
\item Hyponormal operators.\\
Since $iP_T$ is positive, we have $$||[T^*,T]||_1 = \tr([T^*,T])= |2i\tr([X,Y])|= 2\left|\int dP_T\right|=2||P_T||.$$
An example of hyponormal almost normal operator is the Ces\`{a}ro operator $C_0: \ell^2(\mathbb{N}) \to \ell^2(\mathbb{N})$ defined by
$$ (a_1,a_2, \cdots) \mapsto \left(a_1, \frac{a_1+a_2}{2}, \cdots, \frac{a_1+a_2+\cdots+a_n}{n},\cdots \right).$$ 
Spectral properties of $C_0$ is explained in \cite{brownhalmosshields} and almost normality is shown in the page 660 of \cite{kittanehcesaro}.
\item Almost normal weighted shifts.\\
Let $W_\alpha$ be a bilateral weighted shift operator on $\ell^2(\mathbb{Z})$ with weights $\{\alpha_n\}_{n=-\infty}^\infty$. Namely, $W_\alpha e_n=\alpha_n e_{n+1}$ for the canonical orthonormal basis $\{e_n\}_{n=-\infty}^\infty$ of $\ell^2(\mathbb{Z})$.
Since the Helton--Howe measure is unitarily invariant, we may assume every $\alpha_n$ is nonnegative and real. Then, $W_\alpha$ is almost normal if and only if
$$\sum_{n=-\infty}^\infty |\alpha_n^2-\alpha_{n+1}^2|<\infty.$$
In particular, $\alpha_+=\lim_{n\to \infty}\alpha_n$ and $\alpha_-=\lim_{n\to -\infty}\alpha_n$ exist if $W_\alpha$ is almost normal.
We only consider the case $\alpha_-\leq\alpha_+$. The other case can be dealt with similarly. In our case, $\sigma(W_\alpha)$ is the annulus $\mathrm{ann}(\alpha_-,\alpha_+)$ with the outer radius $\alpha_+$ and the inner radius $\alpha_-$ and the map $\mathbb{C}\setminus\sigma_e(W_\alpha)\ni\lambda \mapsto \ind(\lambda-W_\alpha)$ takes the value $-1$ on the interior of this annulus and $0$ on the exterior of the annulus. This spectral result can be found in Proposition 27.7 of \cite{conway2000course}.
Thus 
$$||P_{W_\alpha}|| =\frac{1}{2\pi}\mathrm{Area}(\mathrm{ann}(\alpha_-,\alpha_+))=\frac{\alpha_+^2-\alpha_-^2}{2}.$$
For a natural number $k$, let $T_k$ be the finite rank operator so that $W_\alpha+T_k$ is the weighted shift operator 
with weights $\{\beta_n^k\}_{n=-\infty}^\infty$ defined by
\begin{equation*}
\beta_n^k=
\begin{cases*}
\alpha_+ & if $ 0 \leq n\leq k$, \\
\alpha_- & if $-k \leq n <0$, \\
\alpha_n & else. 
\end{cases*}
\end{equation*}
Then,
$$||[(W_\alpha+T_k)^*,W_\alpha+T_k]||_1=|\alpha_+^2-\alpha_-^2|+|\alpha_+^2-\alpha_{k+1}^2|+|\alpha_-^2-\alpha_{-k-1}^2|+\sum_{n\leq -k-1,k\leq n}|\alpha_n^2-\alpha_{n+1}^2|.$$
So we obtain
$$\inf_{W_\alpha-T'\in\mathcal{J}_1}||[(T')^*,T']||_1 \leq \lim_{k\to\infty}||[(W_\alpha+T_k)^*,W_\alpha+T_k]||_1=\alpha_+^2-\alpha_-^2=2||P_{W_\alpha}||.$$
\end{enumerate}

\section{The Helton--Howe Measure of Toeplitz Operators}

\subsection{Toeplitz Operators with Smooth Symbols}

\begin{comment}
We start with a lemma that plays an important role in our computation.
The first is Izumi's trace formula.

\begin{lemma}[Lemma 2.3 of~\cite{izumi2025generalized}]\label{izumitraceformula}
Let $f,g\in C^\infty(\mathbb{T})$, and let $h\in L^\infty(\mathbb{T})$. Then we have
$$\mathrm{tr} (T_h[T_f,T_g])=\frac{1}{2\pi i}\int_{\mathbb{D}}H dF\land dG.$$
\end{lemma}

Indeed, Izumi allowed $f,g$ to be in the Krein algebra. But we only need the result for smooth functions. Next is an elementary fact.

\begin{lemma}\label{trace equality}
If $f,g,h$ are elements of $C^\infty(\mathbb{T})$, then
$$ \mathrm{tr}([T_fT_g, T_h])=\mathrm{tr}([T_{fg}, T_h]).$$
\end{lemma}

\begin{proof}
Since $f$ and $g$ belong to $C^\infty(\mathbb{T})$, one has $T_f T_g-T_{fg}=-H_{\bar{f}}^*H_g \in \mathcal{J}_1.$ It follows from Proposition \ref{properties of tr} that the lemma is proven.
\end{proof}
\end{comment}

Let $\phi\in C^\infty(\mathbb{T})$. Then, the essential spectrum of $T_\phi$ is $\phi(\mathbb{T})$, whose two dimensional Lebesgue measure is $0$. Hence by Theorem \ref{a.c.},
$$dP_{T_\phi}=-\frac{1}{2\pi i} \ind(T_\phi-x-iy)dxdy,$$
where we regard $\lambda \mapsto \ind(T_\phi-\lambda)$ as an almost everywhere defined function. Note that $\ind(T_\phi-\lambda)=-\mathrm{wind}(\phi, \lambda)$ (c.f. Theorem 4.4.3 of \cite{arvesonspectral}). 
Identifying the winding number with the degree of the mapping $\Phi :\mathbb{D}\to \mathbb{R}^2$, it is immediate that $\mathrm{wind}(\phi,\lambda)=m_\Phi(\lambda)$.
 A proof of this fact can be found in \S 6.6 in Chapter 1 of \cite{deimling1985}.
As a consequence, we have the following result.

\begin{theorem}\label{subtheorem}
Let $\phi  \in C^\infty(\mathbb{T})$. Then, the Helton--Howe measure of $T_\phi$ is expressed as follows.
$$dP_{T_\phi}=\frac{1}{2\pi i}m_\Phi(x+iy)dxdy.$$
\end{theorem}

\subsection{Main Results}
We calculate the Helton--Howe measure of general almost normal Toeplitz operators using the result of the previous section. The main theorem of this paper is the following. 

\begin{theorem}[Main Theorem]\label{general formula}
Let $\phi$ be an $L^\infty$ function on the unit circle.
If $T_\phi$ is almost normal, its Helton--Howe measure is expressed as 
$$dP_{T_{\phi}}=\mathrm{weak}^*\text{-}\lim_{r \nearrow1} dP_{T_{\phi_r}}=\mathrm{weak}^*\text{-}\lim_{r \nearrow1} \frac{1}{2\pi i}m_{\Phi_r}(x+iy)dxdy.$$
If, moreover, $\int _\mathbb{D}|J(\Phi)|dxdy<\infty$, then
$$dP_{T_{\phi}}=\frac{1}{2\pi i}m_{\Phi}(x+iy)dxdy.$$
\end{theorem}

We will deduce this theorem from a series of  lemmata and propositions below.

\begin{definition}
Fix a positive number $0<r<1$ and an orthonormal basis $\{e_n\}_{n=0}^\infty$ of our Hilbert space. Then, define an operator $R$ to be the diagonal operator 
$$R=\mathrm{diag}((r^n)_{n=0}^\infty) $$
with respect to the basis $\{e_n\}_{n=0}^\infty$.
\end{definition}

\begin{definition}
Fix an integer $\ell\geq 2$ and an orthonormal basis $\{e_n\}_{n=0}^\infty$ of our Hilbert space. For a bounded operator $X$, define $X^{(\ell)}$ to be the operator 
$$X^{(\ell)}=S^{\ell-2} X(S^*)^{\ell-2}$$
where $S$ is the unilateral shift operator with respect to the basis $\{e_n\}_{n=0}^\infty$.
Its matrix representation may help understand what it is;
$$
X^{(\ell)}= \left(
\begin{array}{c:c}
\begin{matrix}
O_{\ell-2,\ell-2}
\end{matrix} & \text{\rule{0pt}{0pt}O} \\
\hdashline 
 \text{\rule{0pt}{17pt}O}  &  \textit{$X$}
\end{array}
\right).
$$
\end{definition}

We also need a matrix representation of the commutator of Toeplitz operators.
Let $\{e_n\}_{n=-\infty}^\infty$ be the canonical orthonormal basis of $L^2(\mathbb{T})$. 

\begin{lemma}\label{matrix entry of toeplitz commutator}
Let $\phi=f+\bar{g}\in L^\infty(\mathbb{T})$ with $f,g \in H^2(\mathbb{T})$ and let $n,m$ be nonnegative integers. Then, 
$$\langle [T_\phi^*,T_\phi]e_n,e_m\rangle=\sum_{\ell>0}\left(\hat{f}(m+\ell)\overline{\hat{f}(n+\ell)}-\hat{g}(m+\ell)\overline{\hat{g}(n+\ell)}\right).$$
\end{lemma}
\begin{proof}
By a direct computation, one has
$[T_\phi^*,T_\phi]=H_{\bar{\phi}}^*H_{\bar{\phi}}-H_\phi^*H_\phi$ and
\begin{align*}
\langle H_{\bar{\phi}}^*H_{\bar{\phi}}e_n, e_m \rangle &=\langle (1-P_+)\bar{\phi}e_n,(1-P_+)\bar{\phi}e_m \rangle\\
&=\sum_{\ell\in\mathbb{Z}} \langle (1-P_+)\bar{\phi}e_n,e_\ell \rangle \langle e_\ell,(1-P_+)\bar{\phi}e_m \rangle\\
&= \sum_{\ell>0} \langle \bar{\phi}e_n,e_{-\ell} \rangle \langle e_{-\ell},\bar{\phi}e_m \rangle\\
&=\sum_{\ell>0}\hat{f}(m+\ell)\overline{\hat{f}(n+\ell)}.
\end{align*}
We can similarly compute $\langle H_{\phi}^*H_{\phi} e_n,e_m\rangle$ and the lemma is proven.
\end{proof}

For $n>0$, let $P_n$ be the orthogonal projection onto $\text{span}\{e_0,e_1, \cdots, e_n \}$.
We will use the following lemma concerning the approximation by specific finite rank operators. Though we will deal with general normed ideals, only the case where $\mathcal{J}_\Psi^{(0)}=\mathcal{K}(\mathcal{H})$ is needed to prove the main theorem of this paper. 

\begin{lemma}\label{sot approximation}
Let K be an operator in $\mathcal{J}_\Psi^{(0)}$. If $\{T_i\}$ is an operator-norm-bounded net converging to $T$ in the SOT, then $\{T_i K\}$ converges to $TK$ in $\mathcal{J}_\Psi^{(0)}$.
\end{lemma}

A proof of this lemma for Schatten $p$-classes can be found in Corollary 1.18 of \cite{nestalgebras}. One can imitate the technique to give a proof for $\mathcal{J}_\Psi^{(0)}$.

\begin{remark}\label{remark on finite matrix}
If we take our Hilbert space to be $H^2(\mathbb{T})$, $P_nXP_n$ converges to $X$ in $\mathcal{J}_\Psi^{(0)}$ if $X\in \mathcal{J}_\Psi^{(0)}$ as stated in Lemma \ref{sot approximation}. So we often consider an $X=(x_{n,m})_{n,m=0}^\infty$ with finitely many nonzero entries with respect to $\{e_n\}_{n=0}^\infty$. 
\end{remark}

\begin{lemma}\label{extracting r preliminary}
Let $\phi\in L^\infty(\mathbb{T})$ and $X$ be a finite rank operator as in Remark \ref{remark on finite matrix}. Then, $\tr([T_{\phi_r}^*,T_{\phi_r}]X)$ equals
$$r^2 \tr([T_{\phi}^*,T_{\phi}]RXR)-\sum_{\ell\geq 2} r^{2\ell-2}(1-r^2)~\tr([T_{\phi}^*,T_{\phi}](RXR)^{(\ell)}).$$
\end{lemma}
\begin{proof}
We write $\phi=f+\bar{g}$ where $f,g\in H^2(\mathbb{T})$.
Let us define $$S(\phi)_{m,n}:=\sum_{\ell>0}\left(\hat{f}(m+\ell)\overline{\hat{f}(n+\ell)}-\hat{g}(m+\ell)\overline{\hat{g}(n+\ell)}\right).$$ By Lemma \ref{matrix entry of toeplitz commutator}, this is the $(m,n)$-entry of $ [T_\phi^*,T_\phi]$. 
The $(m,n)$-entry of $ [T_{\phi_r}^*,T_{\phi_r}]$ is obtained by replacing $\hat{f}(k)$ and $\hat{g}(k)$ with $r^k\hat{f}(k)$ and $r^k\hat{g}(k)$, respectively.
Note that 
$S(\phi)_{m+\ell-1,n+\ell-1}-S(\phi)_{m+\ell,n+\ell}=\hat{f}(m+\ell)\overline{\hat{f}(n+\ell)}-\hat{g}(m+\ell)\overline{\hat{g}(n+\ell)}$.
Thus by Lemma \ref{matrix entry of toeplitz commutator}, $\tr([T_{\phi_r}^*,T_{\phi_r}]X)$ is equal to
\begin{align*}
&~~~~\sum_{m\geq 0}\left(\sum_{n\geq0} S(\phi_r)_{m,n}x_{n,m}\right)\\
&=\sum_{m\geq 0}\left(\sum_{n\geq0} x_{n,m}\sum_{\ell>0}r^{n+m+2\ell}\left(\hat{f}(m+\ell)\overline{\hat{f}(n+\ell)}-\hat{g}(m+\ell)\overline{\hat{g}(n+\ell)}\right)\right)\\
&=\sum_{m\geq 0}\left(\sum_{n\geq0} x_{n,m}\sum_{\ell>0}r^{n+m+2\ell}(S(\phi)_{m+\ell-1,n+\ell-1}-S(\phi)_{m+\ell,n+\ell})\right)\\
&=\sum_{m\geq 0}\left(\sum_{n\geq0} r^{n+m}x_{n,m}\left(r^2S(\phi)_{m,n}+\sum_{\ell\geq2}r^{2\ell-2}(r^2-1)S(\phi)_{m+\ell-1,n+\ell-1}\right)\right)\\
&= \underset{\mathrm{I}}{\underline{r^2 \sum_{m\geq 0}\left(\sum_{n\geq0} r^{n+m}x_{n,m}S(\phi)_{m,n}\right) }}+ \underset{\mathrm{II}}{\underline{\sum_{m\geq 0}\left(\sum_{n\geq0} r^{n+m}x_{n,m}\sum_{\ell\geq2}r^{2\ell-2}(r^2-1)S(\phi)_{m+\ell-1,n+\ell-1}\right) }}.
\end{align*}
Using Lemma \ref{matrix entry of toeplitz commutator} again, we see that
$$\mathrm{I} =r^2 \sum_{m\geq 0}\left(\sum_{n\geq0} (RXR)_{n,m}S(\phi)_{m,n}\right) =r^2 \tr([T_{\phi}^*,T_{\phi}]RXR).$$
Meanwhile, 
\begin{align*}
\mathrm{II} &=\sum_{m\geq 0}\left(\sum_{n\geq0} (RXR)_{n,m}\sum_{\ell\geq2}r^{2\ell-2}(r^2-1)S(\phi)_{m+\ell-1,n+\ell-1}\right) \\
&=\sum_{m\geq 0}\left(\sum_{\ell \geq2} r^{2\ell-2}(r^2-1)\sum_{n\geq0}(RXR)_{n,m}S(\phi)_{m+\ell-1,n+\ell-1}\right) \\
&= \sum_{\ell \geq2} r^{2\ell-2}(r^2-1) \sum_{m\geq 0}\left(\sum_{n\geq0}(RXR)_{n,m}S(\phi)_{m+\ell-1,n+\ell-1}\right) \\
&=\sum_{\ell \geq 2}r^{2\ell-2}(r^2-1)  \sum_{m \geq \ell-1} \left(\sum_{n\geq \ell-1}(RXR)_{n-\ell+1,m-\ell+1}S(\phi)_{m,n}\right) \\
&=\sum_{\ell \geq 2}r^{2\ell-2}(r^2-1)  \sum_{m \geq \ell-1} \left(\sum_{n\geq \ell-1}((RXR)^{(\ell)})_{n,m}S(\phi)_{m,n}\right). 
\end{align*}
Changing the order of summations above is valid since $X$ is a finite matrix. Since the $(m,n)$-entry of $(RXR)^{(\ell)}$ is $0$ if either $m$ or $n$ is less than $\ell-1$, one has
\begin{align*}
\mathrm{II}
&=\sum_{\ell \geq 2}r^{2\ell-2}(r^2-1)  \sum_{m \geq 0} \left(\sum_{n\geq 0}((RXR)^{(\ell)})_{n,m}S(\phi)_{m,n}\right) \\
&=\sum_{\ell \geq2} r^{2\ell-2}(r^2-1)~\tr([T_{\phi}^*,T_{\phi}](RXR)^{(\ell)}).
\end{align*}
Thus we have the desired result.
\end{proof}

If $T_\phi$ is almost normal, we can refine the above lemma.

\begin{lemma}\label{extracting r}
Let $\phi\in L^\infty(\mathbb{T})$ and $X$ be any bounded operator. Suppose that $T_\phi$ is almost normal. Then, $\tr([T_{\phi_r}^*,T_{\phi_r}]X)$ equals
$$r^2 \tr([T_{\phi}^*,T_{\phi}]RXR)-\sum_{\ell\geq 2} r^{2\ell-2}(1-r^2)~\tr([T_{\phi}^*,T_{\phi}](RXR)^{(\ell)}).$$
\end{lemma}
\begin{proof}
 For a bounded operator $X$, let $X_n=P_nXP_n$.
Then, Lemma \ref{extracting r preliminary} implies
\begin{align}\label{trace decomposition}
&~~~~\tr([T_{\phi_r}^*,T_{\phi_r}]X_n) \notag \\
&=r^2 \tr([T_{\phi}^*,T_{\phi}]RX_nR)-\sum_{\ell\geq 2} r^{2\ell-2}(1-r^2)~\tr([T_{\phi}^*,T_{\phi}](RX_nR)^{(\ell)}).
\end{align}
Since $X_n$ converges to $X$ in the SOT as $n\to \infty$, it follows from Lemma \ref{sot approximation} that
$$\lim_{n\to \infty}\tr([T_{\phi_r}^*,T_{\phi_r}]X_n)=\tr([T_{\phi_r}^*,T_{\phi_r}]X), ~~\lim_{n\to \infty}\tr([T_{\phi}^*,T_{\phi}]RX_nR)=\tr([T_{\phi}^*,T_{\phi}]RXR)$$ and that
$$\lim_{n\to \infty}\tr([T_{\phi}^*,T_{\phi}](RX_nR)^{(\ell)})=\tr([T_{\phi}^*,T_{\phi}](RXR)^{(\ell)})$$
for all $\ell\geq 2.$ Noting that $|\tr([T_{\phi}^*,T_{\phi}](RX_nR)^{\ell})|\leq ||[T_{\phi}^*,T_{\phi}]||_1||X||,$ Lebesgue's dominated convergence theorem yields
\begin{align*}
&\lim_{n\to\infty} \sum_{\ell\geq 2} r^{2\ell-2}(1-r^2)~\tr([T_{\phi}^*,T_{\phi}](RX_nR)^{(\ell)})\\
&= \sum_{\ell\geq 2} r^{2\ell-2}(1-r^2)~\tr([T_{\phi}^*,T_{\phi}](RXR)^{(\ell)}).
\end{align*}
Taking limits on both sides of Equation (\ref{trace decomposition}), we have the desired trace formula.
\end{proof}

The point of the above two lemmata is that we can extract $r$ from the symbol of the Toeplitz operator. The next proposition is also an essential tool. Again casual readers can only consider the case where $\mathcal{J}_{\Psi^*}=\mathcal{J}_1$ and $\mathcal{J}_\Psi^{(0)}=\mathcal{K}(\mathcal{H})$.

\begin{proposition}\label{approximation of almost normal operators}
Let $\phi\in L^\infty(\mathbb{T})$ and $\Psi$ be a symmetric norm. Assume that $[T_{\phi}^*, T_{\phi} ]\in \mathcal{J}_{\Psi^*}$. Then, for every $r<1$,
$$||[T_{\phi_r}^*, T_{\phi_r} ]||_{\mathcal{J}_{\Psi^*}}\leq 2||[T_{\phi}^*, T_{\phi} ]||_{\mathcal{J}_{\Psi^*}}.$$
\end{proposition}
\begin{proof}
Since finite rank operators are dense in $\mathcal{J}_\Psi^{(0)}$, a bounded operator $T$ is in $\mathcal{J}_{\Psi^*}=(\mathcal{J}_\Psi^{(0)})^*$ if and only if there exists a constant $M>0$ satisfying $|\tr(TX)| \leq M||X||_{\mathcal{J}_\Psi^{(0)}}$ for all finite rank operators $X$. The minimum of such $M$'s is $||T||_{\mathcal{J}_{\Psi^*}}$. By Lemma \ref{sot approximation}, we only consider an $X$ as in Remark \ref{remark on finite matrix}. Then, Lemma \ref{extracting r preliminary} yields
$$\tr([T_{\phi_r}^*,T_{\phi_r}]X)=r^2 \tr([T_{\phi}^*,T_{\phi}]RXR)-\sum_{\ell\geq 2} r^{2\ell-2}(1-r^2)~\tr([T_{\phi}^*,T_{\phi}](RXR)^{(\ell)}).$$
Noting that $||RXR||_{\mathcal{J}_\Psi^{(0)}}$ and $||(RXR)^{(\ell)}||_{\mathcal{J}_\Psi^{(0)}}$ are bounded by $||X||_{\mathcal{J}_\Psi^{(0)}}$, one has the desired result.
\end{proof}

\begin{proposition}\label{sufficient condition}
Let $\phi\in L^\infty(\mathbb{T})$. If $\{||[T_{\phi_r}^*, T_{\phi_r} ]||_1\}_{r<1}$ is uniformly bounded by some $M>0$, then $T_{\phi}$ is almost normal and 
$$dP_{T_{\phi}}=\mathrm{weak}^*\text{-}\lim_{r \nearrow1} \frac{1}{2\pi i}m_{\Phi_r}(x+iy)dxdy.$$
\end{proposition}
\begin{proof}
Since $\phi_r$ converges to $\phi$ almost everywhere on $\mathbb{T}$ as $r\nearrow 1$ with the $L^\infty$-norm uniformuly bounded (c.f. maximum modulus principle), $T_{\phi_r}$ converges to $T_\phi$ with respect to the strong-$*$ operator topology. Since the unit ball of $\mathcal{J}_1$ is closed with respect to the WOT (Theorem 2.7 (d) of \cite{simon2005trace}), it follows $[T_{\phi}^*, T_{\phi} ]$ is in the trace class and that $||[T_{\phi}^*, T_{\phi} ]||_1\leq M$ under the assumption of the proposition.

To express $dP_T$ as a weak-$*$ limit, we have to show
$$\text{tr}([T_{\Re{\phi}}^n T_{\Im{\phi}}^m, T_{\Im{\phi}}])=\lim_{r\nearrow 1} \text{tr}([T_{\Re{\phi}_r}^n T_{\Im{\phi_r}}^m, T_{\Im{\phi_r}}]).$$
Using the formula $[AB,C]=A[B,C]+[A,C]B$ for any bounded operators $A,B,C$, we can write
$$[T_{\Re{\phi}_r}^n T_{\Im{\phi_r}}^m, T_{\Im{\phi_r}}]=\sum_{j:\mathrm{finite}} A_r^{(j)}[T_{\Re{\phi}_r}, T_{\Im{\phi_r}}]B_r^{(j)},$$
where $A_r^{(j)}$ and $B_r^{(j)}$ are products of $T_{\Re{\phi}_r}$ and $T_{\Im{\phi_r}}$.
Then, 
$$\tr([T_{\Re{\phi}_r}^n T_{\Im{\phi_r}}^m, T_{\Im{\phi_r}}])=\sum_{j:\mathrm{finite}} \tr(A_r^{(j)}[T_{\Re{\phi}_r}, T_{\Im{\phi_r}}]B_r^{(j)})=\sum_{j:\mathrm{finite}} \tr([T_{\Re{\phi}_r}, T_{\Im{\phi_r}}]C_r^{(j)}),$$
where $C_r^{(j)}=B_r^{(j)}A_r^{(j)} $.
Note that each $C_r^{(j)}$ converges in the SOT to an operator obtained by replacing $\phi_r$ appearing in $C_r^{(j)}$ with $\phi$. We denote $C^{(j)}:=s\text{-}\lim_{r\nearrow 1}C_r^{(j)}$.
It follows from the definition of $C^{(j)}$ that
$$\tr([T_{\Re{\phi}}^n T_{\Im{\phi}}^m, T_{\Im{\phi}}])=\sum_{j:\mathrm{finite}} \tr([T_{\Re{\phi}}, T_{\Im{\phi}}]C^{(j)}).$$
(Here, we are not taking limits, but just replacing $\phi_r$ with $\phi$.)
So we have to show 
$$\text{tr}([T_{\Re{\phi}}, T_{\Im{\phi}}]C^{(j)})=\lim_{r\nearrow 1} \text{tr}([T_{\Re{\phi}_r}, T_{\Im{\phi_r}}]C_r^{(j)}).$$
Recall that by Lemma \ref{extracting r}, $\tr([T_{\Re{\phi}_r},T_{\Im{\phi}_r}]C_r^{(j)})$ is equal to $$r^2 \tr([T_{\Re{\phi}},T_{\Im{\phi}}]RC_r^{(j)}R)-\tr\left(\sum_{\ell\geq 2} r^{2\ell-2}(1-r^2)[T_{\Re{\phi}},T_{\Im{\phi}}](RC_r^{(j)}R)^{(\ell)}\right).$$
Now, $\lim_{r\nearrow 1}\tr([T_{\Re{\phi}_r},T_{\Im{\phi}_r}]C_r^{(j)})$ coincides with $$r^2 \lim_{r\nearrow 1}\tr([T_{\Re{\phi}},T_{\Im{\phi}}]RC_r^{(j)}R)-\lim_{r\nearrow 1}\tr\left([T_{\Re{\phi}},T_{\Im{\phi}}]\sum_{\ell\geq 2} r^{2\ell-2}(1-r^2)(RC_r^{(j)}R)^{(\ell)}\right).$$
Since $C^{(j)}=s\text{-}\lim_{r\nearrow 1}C_r^{(j)}$ with the operator norm uniformly bounded, we have $s\text{-}\lim_{r\nearrow 1}RC_r^{(j)}R=C^{(j)}$. Moreover, $Y_r:=\sum_{\ell\geq 2} r^{2\ell-2}(1-r^2)(RC_r^{(j)}R)^{(\ell)}$ converges to $0$ in the SOT as shown below.
Note that the sum $Y_r$ is norm convergent with the norm uniformly bounded. To verify the SOT convergence, take an $x\in \mathrm{span}\{e_0,e_1,e_2,\cdots\}$. Then, $x_r:=Y_r x$ is in fact a finite sum. Thus $\lim_{r \nearrow 1}x_r=0.$
For an arbitrary $y\in H^2(\mathbb{T})$, approximate it by elements in $\mathrm{span}\{e_0,e_1,e_2,\cdots\}$. Hence we obtain $\lim_{r\nearrow 1} Y_r y=0$ as claimed. This combined with Lemma \ref{sot approximation} yields $\lim_{r\nearrow 1}\tr([T_{\Re{\phi}_r},T_{\Im{\phi}_r}]C_r^{(j)})=\tr([T_{\Re{\phi}},T_{\Im{\phi}}]C^{(j)}).$
So far we have shown
\begin{equation}\label{limit formula}
\text{tr}([T_{\Re{\phi}}^n T_{\Im{\phi}}^m, T_{\Im{\phi}}])=\lim_{r\nearrow 1} \text{tr}([T_{\Re{\phi}_r}^n T_{\Im{\phi_r}}^m, T_{\Im{\phi_r}}]).
\end{equation}
We will deduce the trace formula from this equation. Let $K$ be a compact set containing the open ball of radius $||\phi||_{L^\infty}$ centered at $0$.
$K$ contains the spectra of $T_\phi$ and $T_{\phi_r}$.
Since $\phi_r$ is in $C^\infty(\mathbb{T})$, we have the following equations due to Theorem \ref{subtheorem} and Equation (\ref{limit formula}).
\begin{align*}
\int_K nx^{n-1}y^mdP_{T_\phi}&=\text{tr}([T_{\Re{\phi}}^n T_{\Im{\phi}}^m, T_{\Im{\phi}}])\\
&=\lim_{r\nearrow 1} \text{tr}([T_{\Re{\phi}_r}^n T_{\Im{\phi_r}}^m, T_{\Im{\phi_r}}])\\
&=\lim_{r\nearrow 1} \frac{1}{2\pi i}\int_K nx^{n-1}y^m m_{\Phi_r}(x+iy)dxdy.
\end{align*}
That is, for any polynomial $p$, we have
\begin{equation}\label{formula for polynomial}
\int_K p(x,y)dP_{T_\phi}=\lim_{r\nearrow 1} \frac{1}{2\pi i}\int_K p(x,y) m_{\Phi_r}(x+iy)dxdy. 
\end{equation}
Let $f$ be a continuous function on $K$. We are going to validate the next equality
\begin{equation}\label{continuous version}
\int_K f(x,y)dP_{T_\phi}=\lim_{r\nearrow 1} \frac{1}{2\pi i}\int_K f(x,y) m_{\Phi_r}(x+iy)dxdy.\end{equation}
In the following estimates, we will use the fact that the total variation $||P_T||$ is bounded by the half of the trace norm of $[T^*,T]$.
This in turn implies that 
$$||m_{\Phi_r}||_{L^1(K)}=2\pi ||P_{T_{\phi_r}}||\leq \pi||[T_{\phi_r}^*, T_{\phi_r} ]||_1\leq \pi M.$$

Note that since the closed unit ball of $C(K)^*$ is compact with respect to the weak-$*$ topology, $\lim_{r\nearrow 1} \frac{1}{2\pi i}\int_K f(x,y) m_{\Phi_r}(x+iy)dxdy$ exists. Strictly speaking, we have to take a subnet. But we can easily see that the limit doesn't depend on which subnets to take, hence the limit as $r\nearrow 1$ exists. We show that this limit coincides with the integral of $f$ with respect to the Helton--Howe measure. For any $\varepsilon>0$, take a polynomial $p$ such that $||f-p||_{C(K)}<\varepsilon$. One has the following estimates.

\small
\begin{align*}
&~~~~\left|\int_K f~dP_{T_\phi} - \lim_{r\nearrow 1} \frac{1}{2\pi i}\int_K f~ m_{\Phi_r}dxdy\right|\\
&\leq \left|\int_K (f-p)dP_{T_\phi} \right| +\left|\int_K p~dP_{T_\phi}-\lim_{r\nearrow 1} \frac{1}{2\pi i}\int_K p~m_{\Phi_r}dxdy\right| + \left|\lim_{r\nearrow 1} \frac{1}{2\pi i} \int_K (f-p)m_{\Phi_r}dxdy\right|\\
&=\left|\int_K (f-p)dP_{T_\phi} \right| + \left|\lim_{r\nearrow 1} \frac{1}{2\pi i} \int_K (f-p)m_{\Phi_r}dxdy\right|\\
&\leq \varepsilon ||P_{T_\phi}|| + \frac{\varepsilon}{2\pi} \sup_{0<r<1}||m_{\Phi_r}||_{L^1}\\
&\leq \varepsilon ||[T_{\phi}^*, T_{\phi}]||_1 + \varepsilon \sup_{0<r<1}||[T_{\phi_r}^*, T_{\phi_r}]||_1\\
&\leq 2 \varepsilon M.
\end{align*}
\normalsize
Hence we obtain the desired result.
\end{proof}

Imitating the proof of the SOT convergence of $\sum_{\ell\geq 2} r^{2\ell-2}(1-r^2)(RC_r^{(j)}R)^{(\ell)}$ in the above proposition, we obtain the next result.

\begin{corollary}\label{weak convergence}
Let $\phi\in L^\infty(\mathbb{T})$. If $T_\phi$ is almost normal, then $[T_{\phi_r}^*, T_{\phi_r} ]$ converges to $[T_{\phi}^*, T_{\phi} ]$ as $r\nearrow 1$ in the weak topology induced by the duality with $B(\mathcal{H})$.
\end{corollary}

At last, we can prove the main result. 

\begin{proof}[Proof of Theorem \ref{general formula}]
The first statement is immediate from Proposition \ref{sufficient condition} and Proposition \ref{approximation of almost normal operators}.
All we have to prove is the second assertion.
Let $K$ be a compact set containing the ball of radius $||\phi||_{L^\infty}$ centered at $0$. Then,
\begin{align*}
\int_K nx^{n-1}y^mdP_{T_\phi}&=\lim_{r\nearrow 1} \frac{1}{2\pi i}\int_K nx^{n-1}y^m m_{\Phi_r}(x+iy)dxdy \\
&=\frac{1}{2\pi i}\int_K nx^{n-1}y^m m_{\Phi}(x+iy)dxdy.
\end{align*}

We validate the last equality. 
First, note that for $z\in \mathbb{D}$, $\Phi(z)=w$ if and only if $\Phi_r(z/r)=w$. It follows from the definition of the signed multiplicity function that $m_{\Phi_r}$ converges to $m_{\Phi}$ almost everywhere on $K$ and
$$|m_{\Phi_r}(w)|\leq \#(\Phi_r^{-1}(\{w\})\cap \mathbb{D}) \leq \#(\Phi^{-1}(\{w\})\cap\mathbb{D})$$
for all $w\in\mathbb{C}$.
As in Lemma \ref{existence criterion}, $\#(\Phi^{-1}(\{w\})\cap\mathbb{D})$ is integrable. Thus applying Lebesgue's dominated convergence theorem, the last equality holds.
\end{proof}

Next is an immediate consequence of Theorem \ref{index formula} and Theorem \ref{general formula}.

\begin{corollary}
Let $\phi\in L^\infty(\mathbb{T})$ and assume that $T_\phi$ is almost normal and that $\int _\mathbb{D}|J(\Phi)|dxdy<\infty$. Then for almost all $\lambda$ outside $\sigma_e(T_\phi)$, we have
$$\mathrm{ind}(\lambda-T_\phi)=-m_\Phi(\lambda).$$
\end{corollary}

Although we have exhibited a formula for the Helton--Howe measure if $T_\phi$ happens to be almost normal, the author knows only a few conditions to ensure almost normality.
Combining the propositions used to prove the main theorem, we obtain an abstract characterization of almost normality.

\begin{corollary}
$T_\phi$ is almost normal if and only if the net $\{||[T_{\phi_r}^*, T_{\phi_r} ]||_1\}_{r<1}$ is uniformly bounded.
\end{corollary}

One can give a more explicit sufficient condition using Besov spaces. Since $[T_\phi^*,T_\phi]=2i(H_{\Im{\phi}}^*H_{\Re{\phi}}-H_{\Re{\phi}}^*H_{\Im{\phi}})$, we focus on Hankel operators. The following definition can be found in the arguments in page 115 of \cite{zhu2007operator}. 

\begin{definition}[Analytic Besov spaces]
Let $\phi \in H^\infty(\mathbb{T})$  For $1\leq p<\infty$, $\phi$ belongs to the analytic Besov space $A_p$ if
\begin{equation}
\begin{cases*}\notag
 \int_\mathbb{D}(1-|z|^2)^{p-2} |\Phi'(z)|^p dxdy <\infty & \text{if $p \neq1$,} \\
 \int_\mathbb{D} |\Phi''(z)|dxdy <\infty &  \text{if $p=1$}.
\end{cases*}
\end{equation}
\end{definition}

\begin{definition}[Besov spaces]
Let $1\leq p<\infty$. An $L^\infty$ function $\phi$ on $\mathbb{T}$ belongs to the Besov space $B_p$ if $ P_+(\phi)\in A_p$ and $ \overline{(1-P_+)(\phi)}\in A_p.$
\end{definition}

\begin{proposition}[Corollary 1.2 and 2.2 in Chapter 6 of \cite{peller2012hankel} )]
Let $\phi\in L^\infty(\mathbb{T}).$ 
For $1\leq p<\infty$, the Hankel operator $H_\phi$ is in $\mathcal{J}
_p$ if and only if $(1-P_+)\phi\in B_p$.
\end{proposition}

\begin{corollary}
For $\phi\in L^\infty(\mathbb{T})$, $T_\phi$ is almost normal if 
\begin{equation}\notag\Re{\phi}\in B_p ~~~~\text{and}~~~~ \Im{\phi}\in B_q\end{equation}
where $p$ and $q$ are both finite and H\"{o}lder conjugate to each other,
or
\begin{equation}\notag\Re{\phi}\in B_1~~~~\text{and}~~~~ \Im{\phi}\in L^\infty.\end{equation}
\end{corollary}

\begin{lemma}\label{existence of m for B}
Let $\phi\in L^\infty(\mathbb{T})$. Suppose $\Re{\phi}\in B_p$ and $\Im{\phi}\in B_q$ where $p,q>1$ are finite H\"{o}lder conjugate numbers. Then, $T_\phi$ is almost normal and 
\begin{equation}\notag dP_{T_{\phi}}=\frac{1}{2\pi i}m_{\Phi}(x+iy)dxdy.\end{equation}
\end{lemma}

\begin{proof}
By Lemma \ref{existence criterion}, it suffices to show $\int _\mathbb{D}|J(\Phi)|dxdy<\infty$.
Let $f=\Re{\phi}$ and $g=\Im{\phi}$.  $F,G$ denote their harmonic extensions. 
Then,
\begin{align*}
\int _\mathbb{D}|J(\Phi)|dxdy&= \int _\mathbb{D}\left|\left|\frac{\partial \Phi}{\partial z}\right|^2-\left|\frac{\partial \Phi}{\partial \bar{z}}\right|^2\right|dxdy \\
&=\int _\mathbb{D}2 \left| \frac{\partial G}{\partial z} \overline{\frac{\partial F}{\partial z}}-  \frac{\partial F}{\partial z} \overline{\frac{\partial G}{\partial z}}\right|dxdy \\
&\lesssim \left( \int _\mathbb{D} (1-|z|^2)^{p-2}\left|\frac{\partial F}{\partial z}\right|^p dxdy\right)^{\frac{1}{p}}  \left( \int _\mathbb{D} (1-|z|^2)^{q-2}\left|\frac{\partial G}{\partial z} \right|^q dxdy \right)^{\frac{1}{q}}\\ 
& <\infty.
\end{align*}
Note that $\frac{\partial F}{\partial z}$ coincides with the differential of the harmonic extension of $P_+f$. This completes the proof.
\end{proof}

\section*{Acknowledgements}
I am deeply indebted to my supervisor, Professor Masaki Izumi, for his continuous support and guidance. He read earlier drafts of this thesis with great care, and his numerous suggestions have significantly improved the final version of this manuscript.

\bibliographystyle{amsalpha}
\bibliography{references}

\end{document}